\documentclass[a4paper,reqno, 11pt]{amsart}
\usepackage[foot]{amsaddr}
\makeatletter
\let\old@setaddresses\@setaddresses
\def\@setaddresses{\bigskip\bgroup\parindent 0pt\let\scshape\relax\old@setaddresses\egroup}
\makeatother
\makeatletter
\def\thm@space@setup{
\thm@preskip=4mm
\thm@postskip=0mm
}
\makeatother

\usepackage{comment}
\usepackage{amssymb} %Box
\usepackage{amsmath} %Operator
\usepackage{amsthm} %proof
\usepackage[unicode=true]{hyperref}
\hypersetup{
    colorlinks,
    linkcolor={red!50!black},
    citecolor={blue!50!black},
    urlcolor={blue!80!black}
}

\usepackage{todonotes}
\usepackage{bbm}

\newtheorem{theorem}{Theorem}

\newtheorem{lemma}[theorem]{Lemma}

\newtheorem{conjecture}[theorem]{Conjecture}

\newtheorem{question}[theorem]{Question}

\let\leq\leqslant

\let\geq\geqslant

\DeclareMathOperator{\lb}{{\mu}}

\usepackage{color}

\usepackage{subfig}

\graphicspath{{figures/}}

\begin{document}

\title{Minimum maximal matchings in cubic graphs}

\author[W.~Cames van Batenburg]{Wouter Cames van Batenburg}
\address[W.~Cames van Batenburg]{D\'epartement d'Informatique, 
Universit\'e Libre de Bruxelles, 
Brussels, Belgium}
\email{w.p.s.camesvanbatenburg@tudelft.nl}

\thanks{Supported by an ARC grant from the Wallonia-Brussels Federation of Belgium}

\date{\today}
\sloppy

\begin{abstract}
We prove that every connected cubic graph with $n$ vertices has a maximal matching of size at most $\frac{5}{12} n+ \frac{1}{2}$. This confirms the cubic case of a conjecture of Baste, F\"urst, Henning, Mohr and Rautenbach (2019) on regular graphs. More generally, we prove that every graph with $n$ vertices and $m$ edges and maximum degree at most $3$ has a maximal matching of size at most $\frac{4n-m}{6}+ \frac{1}{2}$.
These bounds are attained by the graph $K_{3,3}$, but asymptotically there may still be some room for improvement.
 Moreover, the claimed maximal matchings can be found efficiently.
  As a corollary, we have a $\left(\frac{25}{18} + O \left( \frac{1}{n}\right)\right) $-approximation algorithm for minimum maximal matching in connected cubic graphs.  
\end{abstract}

\maketitle

\section{Introduction}\label{sec:intro}

All graphs in this paper are undirected, finite and simple. A graph $G$ is said to be \emph{subcubic} if $G$ has maximum degree at most $3$, and \emph{cubic} if $G$ is $3$-regular. A subset $M$ of edges of a graph $G$ is a \emph{matching} if no two edges of $M$ are adjacent in $G$. A \emph{maximal matching} is a matching that cannot be extended to a larger matching. %by adding an edge.
 For a graph $G$, let $\gamma(G)$ denote the minimum size of a maximal matching of $G$. The parameter $\gamma(G)$ is sometimes also called the \emph{edge domination number}.

The problem of finding a minimum maximal matching goes back a long way~\cite{GJ79} and is known to be NP-hard, even when restricted to cubic bipartite or cubic planar graphs~\cite{YG80,HK93,DE08,DE08bis} or subcubic induced subgraphs of the grid~\cite{DE13}.
 Approximating the problem within a factor smaller than $7/6$ (in general) or $1+\frac{1}{487}$ (when restricted to cubic graphs) is NP-hard as well~\cite{CC06}. 
%On the other hand, there do exist polynomial-time algorithms for trees~\cite{.}, series-parallel graphs~\cite{.} and graphs of bounded clique width~\cite{.}.
%On the other hand, every maximal matching $M$ provides a $2$-approximation for $\gamma(G)$ since every edge in a minimum maximal matching can cover at most two edges of $M$. 
 
On the other hand, it is well known that every maximal matching of $G$ provides a $2$-approximation for $\gamma(G)$.  
There also exist efficient algorithms that approach $\gamma(G)$ within a factor strictly smaller than $2$, see e.g.~\cite{CLLLM05,SLL09, GLR09, SV12}, yet most of the works just mentioned focus on the general case or on very dense graphs. 
There also exist polynomial-time algorithms that compute $\gamma(G)$ for more restricted graph classes, e.g. trees~\cite{MH77}, series-parallel graphs~\cite{RP88} and graphs of bounded clique width~\cite{EGW01}.
% Wait, But... bounded clique-width is not necessarily sparse... How to frame this? This story is getting messy.
 In this paper, we instead focus on graphs with bounded degree.
%In this paper, we instead study the problem in subcubic graphs, which are on the sparse side of the spectrum. 
Concretely, we study the following recent conjecture of Baste, F\"urst, Henning, Mohr and Rautenbach.

\begin{conjecture}[\cite{BFHMR19}]\label{conj:regulargraphs}
Let $G$ be a connected $\Delta$-regular graph of order $n$ for some $\Delta \geq 3$, then 
$$\gamma(G) \leq \frac{2\Delta-1}{4\Delta}n + \frac{1}{2}.$$
\end{conjecture}

For every $\Delta \geq 3$, the authors of~\cite{BFHMR19} provided evidence for Conjecture~\ref{conj:regulargraphs} in the form of a weaker upper bound which for cubic graphs (i.e.~ $\Delta=3$) specialises to $\gamma(G) \leq \frac{9}{20} n+ \frac{3}{10}$. 
Furthermore, for cubic graphs they confirmed Conjecture~\ref{conj:regulargraphs} under the additional conditions that the graph is bipartite and does not contain a certain $6$-vertex tree as an induced subgraph. 
The $\frac{9n}{20} +O(1)$ bound for cubic graphs was previously derived by Duckworth and Wormald through a linear programming analysis of greedy algorithms on cubic graphs~\cite{DW10}. With a more problem-specific approach, Zwo\'zniak~\cite{Z06} obtained the better bound $\gamma(G) \leq \frac{4n}{9}+\frac{1}{3}$.

%In this paper, we unconditionally prove Conjecture~\ref{conj:regulargraphs} for cubic graphs:
In this paper, we improve these results by proving Conjecture~\ref{conj:regulargraphs} for cubic graphs:

\begin{theorem}\label{thm:vertices}
Let $G$ be a connected cubic graph of order $n$. Then
$$\gamma(G) \leq \frac{5}{12}n + \frac{1}{2},$$
with equality if and only if $G$ is isomorphic to the complete bipartite graph $K_{3,3}$.
\end{theorem}

In doing so, we derive the following slightly more general result, which also applies to non-regular graphs.

\begin{theorem}\label{thm:edges}
Let $G$ be a connected subcubic graph on $n$ vertices and $m$ edges. Then 
$\gamma(G) \leq \frac{4n-m}{6}+ \frac{1}{2}$, with equality if and only if $G$ is isomorphic to $K_{3,3}$.
Moreover, if $G$ is not cubic then
$\gamma(G) \leq \frac{4n-m}{6}.$
\end{theorem}

An effective and widely used approach to prove theorems on subcubic graphs is to take a hypothetical minimum counterexample and then analyse small cut sets as well as the local structure around vertices of small degree, followed by an analysis of the local structure in a highly connected cubic graph, eventually leading to a contradiction~(see e.g.~\cite{CGJ20,BDMNPP19,JRS14} for some recent examples). This paper is no exception. However, one way in which perhaps our approach stands out is how we carefully almost avoid the need to analyse any cubic graph (at the cost of a slightly worse result than is perhaps possible; see the discussion of Conjecture~\ref{conj:K33}). 
This is desirable because the cubic graphs form the hard cases in our context. 
Relatedly, to make the `induction' work, we have devised the following technical theorem, which has Theorems~\ref{thm:vertices} and~\ref{thm:edges} as direct corollaries.

\begin{comment}
\begin{theorem}\label{thm:maincubic}
Let $G$ be a connected graph with maximum degree at most three, having $n_i$ vertices of degree $i$, for each $i \in \left\{1,2,3\right\}$. Let $I$ be the indicator variable being $1$ if $G$ is cubic, and equal to $0$ otherwise. Let $K$ be the indicator variable being $1$ if $G$ is isomorphic to $K_2$, and equal to $0$ otherwise. Then

$$\gamma(G) \leq \frac{5 n_1 + 6 n_2 + 5n_3 + 2K + 6I}{12} =: \lb(G).$$

Moreover, if $G$ is not isomorphic to $K_2$ and has a degree-$1$ vertex $v$ with neighbour $u$, then $G$ has a maximal matching of size at most $\lb(G)$ which avoids the edge $uv$.
\end{theorem}

It will be convenient to use the following equivalent expression for $\lb(G)$. If $G$ has $n$ vertices and $m$ edges, then
%$$\lb(G) = \frac{4n- m - n_1 +K+3I }{6}.$$
$$\lb(G) = \frac{4n- m}{6} + \frac{ K+3I -n_1}{6}.$$
\end{comment}

\begin{theorem}\label{thm:maincubic}
Let $G$ be a connected subcubic graph which is not isomorphic to $K_{3,3}$ and has $n$ vertices, $m$ edges and $n_1$ vertices of degree $1$. Let $I$ be the indicator variable being $1$ if $G$ is cubic, and equal to $0$ otherwise. Let $K$ be the indicator variable being $1$ if $G$ is isomorphic to the complete graph $K_2$, and equal to $0$ otherwise. Then

%$$\gamma(G) \leq \frac{4n- m}{6} + \frac{ 3I + K -n_1}{6} =: \lb(G).$$
$$\gamma(G) \leq \lb(G),$$
where
$$\lb(G):=\frac{4n- m}{6} + \frac{ 2I + K -n_1}{6}.$$

%Moreover, if $G$ is cubic and not isomorphic to $K_{3,3}$ then $\gamma(G) \leq \lb(G) - \frac{1}{6}$.

Moreover, if $G$ is not isomorphic to $K_2$ and has a degree-$1$ vertex $v$ with neighbour $u$, then $G$ has a maximal matching of size at most $\lb(G)$ which avoids the edge $uv$.
\end{theorem}

\textbf{Efficient approximation algorithm}
%\subsection{Efficient approximation algorithm}

The proof of Theorem~\ref{thm:maincubic} is of an inductive nature (formally we will proceed by vertex-minimum counterexample). As such, the proof can be viewed as a polynomial-time algorithm that constructs for every connected subcubic graph $G$ a maximal matching of $G$ of size at most $\lb(G)$. Indeed, we repeatedly identify a set $V_0$ of at most nine vertices and apply the theorem to the smaller graph $G'$ that is obtained from $G$ by deleting $V_0$ and possibly adding one or two edges. Thus we obtain a small maximal matching of $G'$ that subsequently can be completed to a small maximal matching of $G$. In determining the appropriate set $V_0$ and the edges that are to be added back, the only external procedures that we require are (i) find a bridge $e$ (if present) and  determine the number of vertices and edges in each component of $G-e$ and (ii) determine which pairs of neighbours of $V_0$ are connected in $G-V_0$. For both of these subroutines, there exist $O(|V(G)|+|E(G)|)$ time algorithms. The number of iterations required is $O(|V(G)|)$ and therefore in total the required time is at most $O(|V(G)|^2)$.

Until now, we have solely focused on upper bounds on $\gamma(G)$. On the other hand, it is not hard to see that $\gamma(G)\geq \frac{3n}{10}$ for every cubic graph $G$ on $n$ vertices (see e.g. Lemma~3.(i) in~\cite{BFHMR19}). 
Since $\left(\frac{5n}{12}+ \frac{1}{2}\right) / \left(\frac{3n}{10}\right)= \frac{25}{18} n + \frac{5}{3}$, it follows that our proof constitutes a $\left( \frac{25}{18} +  O(\frac{1}{n}) \right)-$approximation algorithm for minimum maximal matching in connected cubic graphs. We reiterate that before this work, the same approximation ratio was already attained in a much more restricted setting, namely for the class of connected bipartite cubic graphs that do not contain a certain $6$-vertex tree as an induced subgraph~\cite{BFHMR19}.

\textbf{Tightness}
%\subsection{Tightness}

Theorem~\ref{thm:edges} is best possible insofar that the complete bipartite graph $K_{3,3}$ and several non-cubic graphs, such as the four-cycle and $K_{3,3}$ with one subdivided edge, attain the bound.
 Likewise, Theorem~\ref{thm:vertices} is best possible due to $K_{3,3}$. Furthermore, there exists a cubic graph on as many as $n=24$ vertices with $\gamma=\frac{5}{12}n$~(e.g. see figure~\ref{fig:spanningK33min}). Nevertheless, we suspect that asymptotically the pre-factor $\frac{5}{12}$ in Theorem~\ref{thm:vertices} is not optimal:

 Let $K_{3,3}^{-}$ denote the graph that is obtained from $K_{3,3}$ by deleting an edge. In~\cite{BFHMR19}, it was conjectured that for cubic connected graphs $G$, it holds that $\gamma(G)= \frac{5n}{12}+\frac{1}{2}$ if and only if $G$ has a spanning subgraph that is the union of an odd number of copies of $K_{3,3}^{-}$. (In particular this would have implied that infinitely many graphs attain the bounds of Theorems~\ref{thm:vertices} and~\ref{thm:edges}.) However, this turns out to be false, as we will now detail. Let $H_1,\ldots, H_k$ be $k\geq 1$ vertex-disjoint copies of $K_{3,3}^{-}$. Writing $p_i, q_i$ for the degree-$2$ vertices of $H_i$, we let $G_k$ be the graph obtained by adding the edges $q_i p_{i+1}$ for all $1 \leq i \leq k$ (where $p_{k+1}$ should be understood to mean $p_1$); see Figure~\ref{fig:spanningK33min}. Then $G_k$ is a cubic connected graph on $n:=6k$ vertices, which however has~\footnote{The equality $\gamma(G_k) = \lceil \frac{7k}{3} \rceil$ follows from an induction applied to $G_{k-3}$; we omit the formal proof. For the upper bound, see the pattern emerging in Figure~\ref{fig:spanningK33min}. To quickly see that $\gamma(G_k)\geq \lfloor \frac{7k}{3} \rfloor$, observe that for every maximal matching $M$ of $G_k$, every three consecutive copies $H_i,H_{i+1},H_{i+2}$ together with the edges $q_ip_{i+1},q_{i+1}p_{i+2},q_{i+2}p_{i+3}$ must contain at least $7$ edges of $M$.} a minimum maximal matching of size $\lceil \frac{7k}{3} \rceil$, so that $\gamma(G_k) =   \frac{\lceil 7k/3\rceil }{6k} n \leq \frac{7}{18} n + \frac{2}{3}$. For every $k\geq 3$, this is smaller than $\frac{5}{12}n$. 
This raises the question whether asymptotically, as the number of vertices increases, the bound of Theorem~\ref{thm:vertices} is improvable.

\begin{comment}
\begin{question}
What is the smallest constant $\rho>0$ such that for every $\epsilon>0$ there exists an $N>0$ such that
for every connected subcubic graph $G$ on $n>N$ vertices,
$$\frac{\gamma(G)}{n} \leq (1+\epsilon) \cdot \rho.$$
Can $\rho$ be as small as $\frac{1}{3}$?
\end{question}
\end{comment}

\begin{question}
What is the infimum $\rho_0$ over all constants $\rho>0$ for which there exists an $N>0$ such that every connected cubic graph $G$ on $n>N$ vertices satisfies
$$\frac{\gamma(G)}{n} \leq \rho?$$
%We know that $\frac{1}{3} \leq \rho_0 \leq \frac{5}{12}$. 
%We know now that $\rho_0 \leq \frac{5}{12}$. Can $\rho_0$ be as small as $\frac{7}{18}$?
\end{question}

%In~\cite{DW10} it was noted that $\rho_0 \geq \frac{3}{8}$. 
The graphs $\left\{G_k\right\}_{k \geq 1}$ and Theorem~\ref{thm:vertices} certify that $\frac{7}{18} \leq \rho_0 \leq \frac{5}{12}$. 
%We think that the lower bound is the right answer and that moreover the following is true.
We suspect that $\rho_0=\frac{7}{18}$ and that moreover the following is true.

\begin{conjecture}\label{conj:betterbound}
Let $G$ be a connected cubic graph of order $n$. Then
$$\gamma(G) \leq \frac{7}{18} n + \frac{2}{3}.$$
\end{conjecture}

\begin{figure}[h]
\centering
\includegraphics[width=0.65\textwidth]{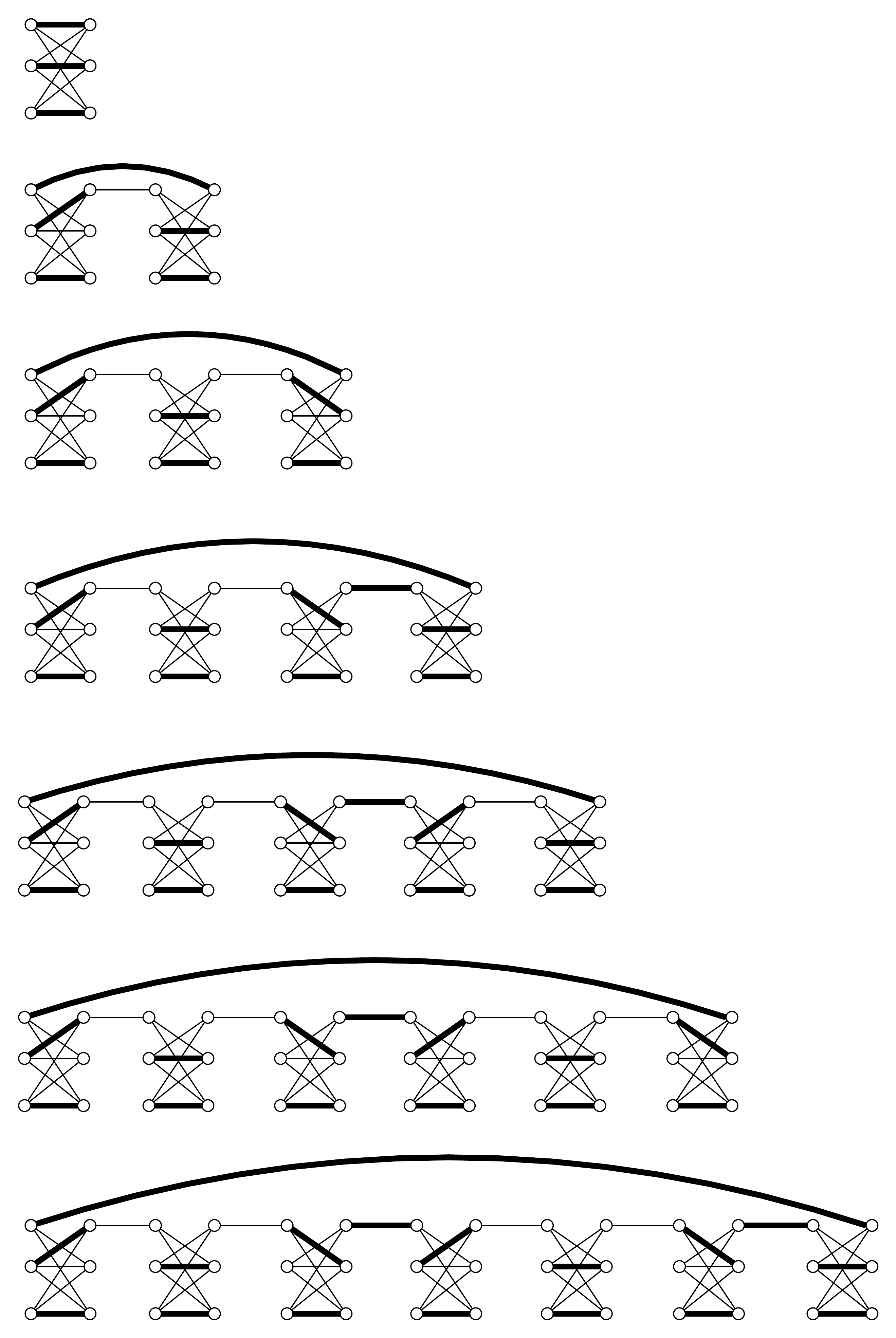}
\caption{Minimum maximal matchings of the graphs $G_1, G_2,\ldots, G_7$.}
\label{fig:spanningK33min}
\end{figure}

In a different direction, we believe the following minor strengthening of Theorem~\ref{thm:edges} is plausible. 

\begin{conjecture}\label{conj:K33}
Let $G$ be a connected subcubic graph with $n$ vertices and $m$ edges. Then $\gamma(G)\leq \frac{4n-m}{6}$ unless $G$ is isomorphic to the complete graph $K_4$ or the complete bipartite graph $K_{3,3}$.
\end{conjecture}

Despite the small gap with Theorem~\ref{thm:edges}, we would not be surprised if considerable more work is required to prove Conjecture~\ref{conj:K33}. This is because in the current paper we were able to almost avoid the analysis of cubic graphs. 
In particular, once we have deduced that our minimum counterexample to Theorem~\ref{thm:maincubic} is cubic, we are done in only a few extra lines, due to the extra room in our bound arising from the additive factor $\frac{I}{3}$.
 This `shortcut' approach turned out to be fruitful but does not seem possible for an attack on Conjecture~\ref{conj:K33}.

\textbf{Some remarks on maximal independent sets and line graphs}
%\subsection{Some remarks on maximal independent sets and line graphs}

Before going to the proof of Theorem~\ref{thm:maincubic}, we would like to point out that the problem of determining $\gamma(G)$ can be viewed as a special case of a more general problem. These remarks can be freely skipped.

First some definitions. An \emph{independent set} of a graph is a set of vertices that are pairwise non-adjacent. A \emph{maximal independent set} is an independent set that cannot be extended to a larger one. Correspondingly, the \emph{independent domination number} $i(G)$ of a graph $G$ is the size of the smallest maximal independent set of $G$. The well-studied \emph{domination number} $\gamma_v(G)$ of $G$ is the size of a smallest subset of $V(G)$ such that every vertex of $G$ is either in or adjacent to it. Finally, the \emph{line graph} $L(G)$ of a graph $G$ is the graph on vertex set $E(G)$ that has an edge between $e_1,e_2\in E(G)$ if and only if $e_1$ and $e_2$ are adjacent in $G$. 

In this terminology, the edge domination number of $G$ can be rewritten as $\gamma(G)= i(L(G))$. Furthermore, for claw-free graphs (and hence for line graphs) it is known~\cite{AllanLaskar78} that the domination number and independent domination number are equal, so that in fact $\gamma(G)= \gamma_v(L(G))$.

%For every cubic graph $G$, its line graph $H$ is $4$-regular and has $3|V(G)|/2$ vertices.
The line graph $L(G)$ of a cubic graph $G$ is $4$-regular and has $3|V(G)|/2$ vertices.
Therefore an equivalent statement of Theorem~\ref{thm:vertices} is that for every ($4$-regular) line graph $L(G)$ of a connected cubic graph $G$, one has $\gamma_v(L(G)) \leq \frac{5\cdot |V(L(G))|}{18}+ \frac{1}{2}$, with equality if and only if $L(G)$ is the line graph of $K_{3,3}$.
One may wonder to what extent the condition of being a line graph is essential for this bound. Often problems on claw-free graphs actually reduce to the case of line graphs~(e.g. see \cite{CaKa19} for a somewhat similar problem where this was the case), so it is natural to expect that more generally $\gamma_v(H) \leq \frac{5 \cdot |V(H)|}{18}+ \frac{1}{2}$ holds for every $4$-regular connected claw-free graph $H$. Assuming Conjecture~\ref{conj:betterbound}, even the stronger bound $\gamma_v(H)\leq\frac{7\cdot |V(H)|}{27}+ \frac{2}{3}$ can be expected. A proof of such a generalization would be nice. 

To put this into perspective, let us now review what is known for a connected graph $H$ on $n$ vertices that is not necessarily claw-free. Then $\gamma_v(H)\leq \frac{3n}{8}$ if $H$ has minimum degree at least three~\cite{Reed96}, $\gamma_v(H)\leq \frac{5n}{14}$ if $n>8$ and $H$ is cubic~\cite{KostochkaStocker09}, and $\gamma_v(H)\leq \frac{4n}{11}$ if $H$ is $4$-regular~\cite{LiuSun04}. 
The first of these two bounds are attained, while for the third bound it is not clear from the literature to what extent it is best possible.
On the other hand, the parameter $i(\cdot)$ can attain much higher values than $\gamma(\cdot)$, as demonstrated by the balanced complete bipartite graphs. There also exists a $4$-regular graph $H$ on $n=14$ vertices with $i(H) =\frac{3n}{7}$ and it is conjectured~\cite{GH13} that this is best possible among all $4$-regular graphs other than $K_{4,4}$.
For cubic graphs more is known in this respect.  For instance, if $H$ is an $n$-vertex connected cubic graph other than $K_{3,3}$ then $i(H)\leq \frac{2n}{5}$, as proved in~\cite{LSS99}, and it is conjectured~\cite{GH13} that $i(H)\leq \frac{3n}{8}$ up to two exceptions, which would be sharp due to infinitely many graphs.
For various other results regarding $\gamma_v(\cdot)$ and $i(\cdot)$, we refer the reader to the survey~\cite{GH13} and citing papers. 
%Also see~\cite{BFHMR_versus20} for a recent conjecture concerning the relation between $\gamma(\cdot)$ and $\gamma_v(\cdot)$.

\section{A technical lemma that helps to avoid cubic graphs}
 
 The following technical lemma will be used a few times in the proof of Theorem~\ref{thm:maincubic}. It allows us to only apply `induction' (between quotation marks because formally we will proceed by minimum counterexample) to non-cubic graphs, which is of considerable benefit due to the factor $I/3$ in our definition of $\lb(G)$. 
 %We present an argument that more generally holds for graphs with maximum degree at most $\Delta$, even though in this paper we will only need the case $\Delta=3$.

\begin{lemma}\label{lem:technicallemma}
Let $G$ be a bridgeless connected subcubic graph. Let $V_0\subseteq V(G)$. Let $N$ denote the set of neighbours of $V_0$ that are not in $V_0$. Let $(N,E^*)$ be another graph on $N$ of which each component is either a singleton or contains at least three vertices. Suppose furthermore that there is at least one edge in $E^*$ which is not an edge of $G$. Then there exists $e\in E^*$ such that $e \notin E(G)$ and the graph $G-V_0+e$ has no cubic component (and is subcubic).
\end{lemma}
\begin{proof}
For each $x\in N$, let $A(x)$ denote the component of $G-V_0$ that contains $x$. Let $x_1x_2\in E^*$ be such that $x_1x_2\notin E(G)$. If possible, we choose $x_1,x_2$ such that additionally $A(x_1)\neq A(x_2)$.   

Note that $N-\left\{x_1,x_2\right\}$ is non-empty, because by assumption, the component of $(N,E^*)$ containing $\left\{x_1,x_2\right\}$ has at least three vertices.

Consider the reduced graph $G':=G-V_0+x_1x_2$ and let $x^*\in N- \left\{x_1,x_2\right\}$. 
Since $x^*$ has a neighbour in $V_0$ in the graph $G$, it has degree less than three in $G'$ and so the component of $x*$ in $G'$ is not cubic.

It remains to show that the component $C$ of $G'$ containing the edge $x_1x_2$ is not cubic. Note that $C$ is the union of $A(x_1)$ and $A(x_2)$ and the edge $x_1x_2$. We may assume that in $G$, the vertex $x_1$ has three neighbours of which exactly one is in $V_0$ (otherwise $x_1$ would have degree less than three in $G'$, implying that $C$ is not cubic); let $v_1$ be the neighbour of $x_1$ that is in $V_0$.

If $A(x_1)\neq A(x_2)$ then the facts that $G$ is connected and bridgeless (in particular $x_1v_1$ is not a bridge) imply that $A(x_1)=A(x^*)$ for some $x^* \in N-\left\{x_1,x_2\right\}$.
On the other hand, if $A(x_1)=A(x_2)$, then by the choice of $x_1$ and $x_2$ we have that $A(x_1)=A(x)$ for all vertices $x \in N$ that belong to the component of $x_1$ in the graph $(N,E^*)$. By assumption, the component of $(N,E^*)$ containing $x_1$ has at least three vertices, so again we obtain $A(x_1)=A(x^*)$ for some $x^* \in N-\left\{x_1,x_2\right\}$.
 In both cases $x^*\in A(x_1)\subseteq C$ and so $C$ is not cubic. 
\end{proof}

Although we do not need it in this paper, note that for any $\Delta\geq 3$, Lemma~\ref{lem:technicallemma} and its proof can be generalised by replacing `cubic' with `$\Delta$-regular' and `subcubic' with `maximum degree at most $\Delta$'.

\section{Proof of the main theorem}

In this section, we prove Theorem~\ref{thm:maincubic}. Let $G$ be a hypothetical counterexample to Theorem~\ref{thm:maincubic} that minimises  $|V(G)|$. We proceed by a series of lemmas that describe properties of $G$ which ultimately lead to a contradiction.

In various cases, we will delete a set of vertices from $G$  and possibly add some edges, and then use that the resulting graph $G'$ is not a counterexample to Theorem~\ref{thm:maincubic}. 
 When the graph $G'$ is clear from the context, we let $n', m'$ and $n_1^{'}$ denote the number of  vertices, edges respectively degree-$1$ vertices of $G'$. Furthermore, $I'$ will denote the number of cubic components of $G'$ and $K'$ will denote the number of components of $G'$ that are isomorphic to $K_2$. We then write $\delta_{n'}:=n-n'$, $\delta_{m'}:=m-m'$, $\delta_{n_1}:= n_1-n_1^{'}$ and $\delta_{I}:= I- I'$  and $\delta_{K}:= K- K'$ for the respective differences.
%Since $\gamma(G') \leq \lb (G')$ as $G'$ is not a counterexample,  we will arrive at a contradiction (as desired) if we can show that $\gamma(G)-\gamma(G') \leq  \lb(G) - \lb(G')$.
%It will therefore always be sufficient to show that

Our main task in deriving a contradiction is to show that $\gamma(G)\leq\lb(G)$. (In fact, this will directly yield the desired contradiction once we know that $G$ has no degree-$1$ vertex.) Since $\gamma(G') \leq \lb (G')$ as $G'$ is not a counterexample, this inequality follows if $\gamma(G)-\gamma(G') \leq  \lb(G) - \lb(G')$.
To conclude that $\gamma(G)\leq \lb(G)$, it will therefore always be sufficient to show that
\begin{equation}\label{ineq:goal}
\delta_{\lb} \geq \delta_{\gamma},
\end{equation}
where
$$\delta_{\lb}:= \lb(G)-\lb(G')=\frac{4 \delta_n -\delta_{m} - \delta_{n_1}  +\delta_{K} + 2\delta_{I}}{6} $$
and
$$\delta_{\gamma}:=\gamma(G)-\gamma(G').$$

\textbf{Remark}
\emph{Since $G$ is connected, whenever only vertices are deleted (and no edge added) to construct $G'$ from $G$, then $G'$ cannot have any cubic component. 
Hence in most of our applications, $\delta_{I}$ will be $0$.}

\textbf{Remark}
\emph{ At some parts in the proof of Theorem~\ref{thm:maincubic}, we deduce that $G$ must be a specific graph on at most nine vertices. 
 In those cases it is easily seen that $G$ is not a counterexample to Theorem~\ref{thm:maincubic}, so we have decided to not explicitly specify a matching certifying this. We remark that we have also performed a small computer search to double-check that indeed no graph on nine vertices or less is a counterexample to Theorem~\ref{thm:maincubic}.}

The bulk of the proof is focused on the structure around vertices of degree smaller than three, starting with the following lemma.

\begin{lemma}\label{lem:nodegreeone}
$G$ has minimum degree at least $2$.
\end{lemma}
\begin{proof}
If there is a vertex of degree $0$, then since $G$ is connected, $G$ must be a singleton, which is not a counterexample to Theorem~\ref{thm:maincubic}.
 %Thus $G$ has minimum degree at least $1$. 
 Next, suppose that $G$ has a vertex $u$ of degree $1$. Let $v$ be the unique neighbour of $u$. Since a star on at most four vertices is not a counterexample either, $v$ must have a neighbour $w$ (distinct from $u$) of degree $\geq 2$. Then consider the graph $G':= G- \left\{u,v,w\right\}$. 
 Let $M'$ be a minimum maximal matching of $G'$.
Then $M:= M' + vw$ is a maximal matching of $G$ of order $\gamma(G')+1$, so $\delta_{\gamma} \leq 1$.  We have  that
$\delta_{n} = 3$ and $\delta_{m} \leq 5$. 
Furthermore, since $G$ is connected, every component of $G'$ that is isomorphic to $K_2$ contains a vertex that has degree larger than $1$ in $G$, and hence we have $\delta_{n_1} \leq 1+\delta_{K}$.

%Thus we arrive at the following contradiction:
Thus
 $$ \delta_{\mu}=\frac{4 \delta_n - \delta_m-  \delta_{n_1} +\delta_{K} +2\delta_{I} }{6} \geq\frac{4 \cdot 3 -5 - (1+\delta_{K}) +\delta_{K} + 0}{6} = 1 \geq \delta_{\gamma},$$

so $M$ must be a maximal matching of $G$ of size at most $\lb(G)$. Because $M$ avoids $uv$, we obtain the desired contradiction.
\end{proof}

Note that by Lemma~\ref{lem:nodegreeone}, $G$ has no degree-$1$ vertex, so from now on \emph{it suffices to derive (\ref{ineq:goal}) to arrive at a contradiction}.

Lemma~\ref{lem:nodegreeone} also implies $n_1=0$, so $\delta_{n_1}=2 \delta_{K}\leq 0$, so $-\delta_{n_1}+\delta_K=\delta_K \geq 0$. Therefore we know from now on that $\delta_{\lb} \geq \frac{4  \delta_n - \delta_m + 2 \delta_{I}}{6}$. Moreover, in all our forthcoming constructions of a reduced graph $G'$, we will make sure that $G'$ does not have any cubic component, so that in fact we will always have
\begin{equation}
\delta_{\lb} \geq \frac{4  \delta_n - \delta_m}{6}.
\end{equation}
This simplifies our computations ever so slightly.

\begin{lemma}\label{lem:nobridge}
$G$ has no bridge
\end{lemma}
\begin{proof}
Suppose $G$ has a bridge $u_0u_1$. For $i\in \left\{0,1\right\}$, let $G_i$ be the component of $G-u_0u_1$ that contains $u_i$. Furthermore, let $H_i$ be the graph induced by $V(G_i) \cup \left\{u_{(1-i)}\right\}$.
Let $n_i:=|V(G_i)|$ and $m_i:=|E(G_i)|$. Observe that $n=n_0+n_1$ and $m=m_0+m_1+1$.

Because $G$ has no degree-$1$ vertex, we have $n_0,n_1>1$. Therefore $H_i$ has fewer vertices than $G$ and hence satisfies Theorem~\ref{thm:maincubic}. For the same reason, $H_i$ cannot be isomorphic to $K_2$. As $u_{(1-i)}$ is a degree-$1$ vertex of $H_i$, it follows that $H_i$ has a maximal matching $M_i$ that avoids $u_0u_1$ and has size at most $\lb(H_i) \leq \lfloor \frac{4(n_i+1) - (m_i+1)-1}{6} \rfloor$.
%So $M:= M_0 \cup M_1$ is a maximal matching of $G$ of size at most $\lfloor \frac{4n_0 - m_0+2}{6} \rfloor + \lfloor \frac{4n_1 - m_1+2}{6} \rfloor.$

On the other hand, because $G_{i}$ satisfies Theorem~\ref{thm:maincubic} as well, $G_i$ has a maximal matching $N_{i}$ of size at most $\lb(G_i) \leq \lfloor  \frac{4n_{i} -m_{i}}{6}\rfloor$. So $\Gamma_i:= N_{1-i} \cup M_{i}$ is a maximal matching of $G$ of size at most $\lfloor  \frac{4n_{1-i} -m_{1-i}}{6}\rfloor + \lfloor \frac{4(n_i+1) - (m_i+1)-1}{6} \rfloor$. Considering the minimum of $|\Gamma_0|$ and $|\Gamma_1|$, we obtain that $G$ has a maximal matching of size at most 

\begin{equation}\label{eq:minimum}
\min \left( \lfloor \frac{4n_1 - m_1}{6} \rfloor+ \lfloor \frac{4n_0 - m_0+2}{6} \rfloor, \lfloor \frac{4n_0 - m_0}{6} \rfloor+ \lfloor \frac{4n_1 - m_1+2}{6} \rfloor \right).
\end{equation}

If at least one of $4n_0-m_0$ and $4n_1-m_1$ does \emph{not} equal $0$ mod $6$, then $\max((4n_0-m_0) \text{ mod } 6, (4n_0-m_0+2) \text{ mod } 6, (4n_1-m_1) \text{ mod } 6, (4n_1-m_1+2) \text{ mod } 6) \geq 3$, so due to the rounding in expression (\ref{eq:minimum}), there is a maximal matching of size at most
\begin{equation*}
\frac{4n_1-m_1 + 4n_0-m_0+2}{6}  - \frac{3}{6}=\frac{4(n_0+n_1)-(m_0+m_1+1)}{6} = \frac{4n-m}{6} \leq \lb(G).
\end{equation*}

Thus $4n_0-m_0=4n_1-m_1= 0 $ mod $6$. Next, for each $i \in \left\{0,1\right\}$, we consider the graph $F_i:= G_i-\left\{u_i\right\}$.
Note that $G_i$ has one more vertex and (depending on the degree of $u_i$) either one or two more edges than $F_i$. So $|V(F_i)|=n_i-1$ while $|E(F_i)|$ either equals $m_i-1$ or $m_i-2$. It follows that 
 $4\cdot |V(F_i)|-|E(F_i)|$  is either $3$ mod $6$ or  $4$ mod $6$.
Hence $F_i$ has a maximal matching $R_i$ of size at most $\lb(F_i)\leq \lfloor \frac{4|V(F_i)| - |E(F_i)| }{6}\rfloor \leq \frac{4|V(F_i)| - |E(F_i)| -3}{6}$. We conclude that $G$ has a maximal matching $R_1+R_2+u_0u_1$ of size at most $\frac{4(|V(F_1)|+|V(F_2)|) - (|E(F_1)|+|E(F_2)|)-2\cdot 3}{6}+1 \leq  \frac{4(n-2)-(m-5)}{6}=\frac{4n-m-3}{6} < \lb(G)$, a contradiction.
\end{proof}

\begin{figure}[h]
\centering
\includegraphics[width=0.55\textwidth]{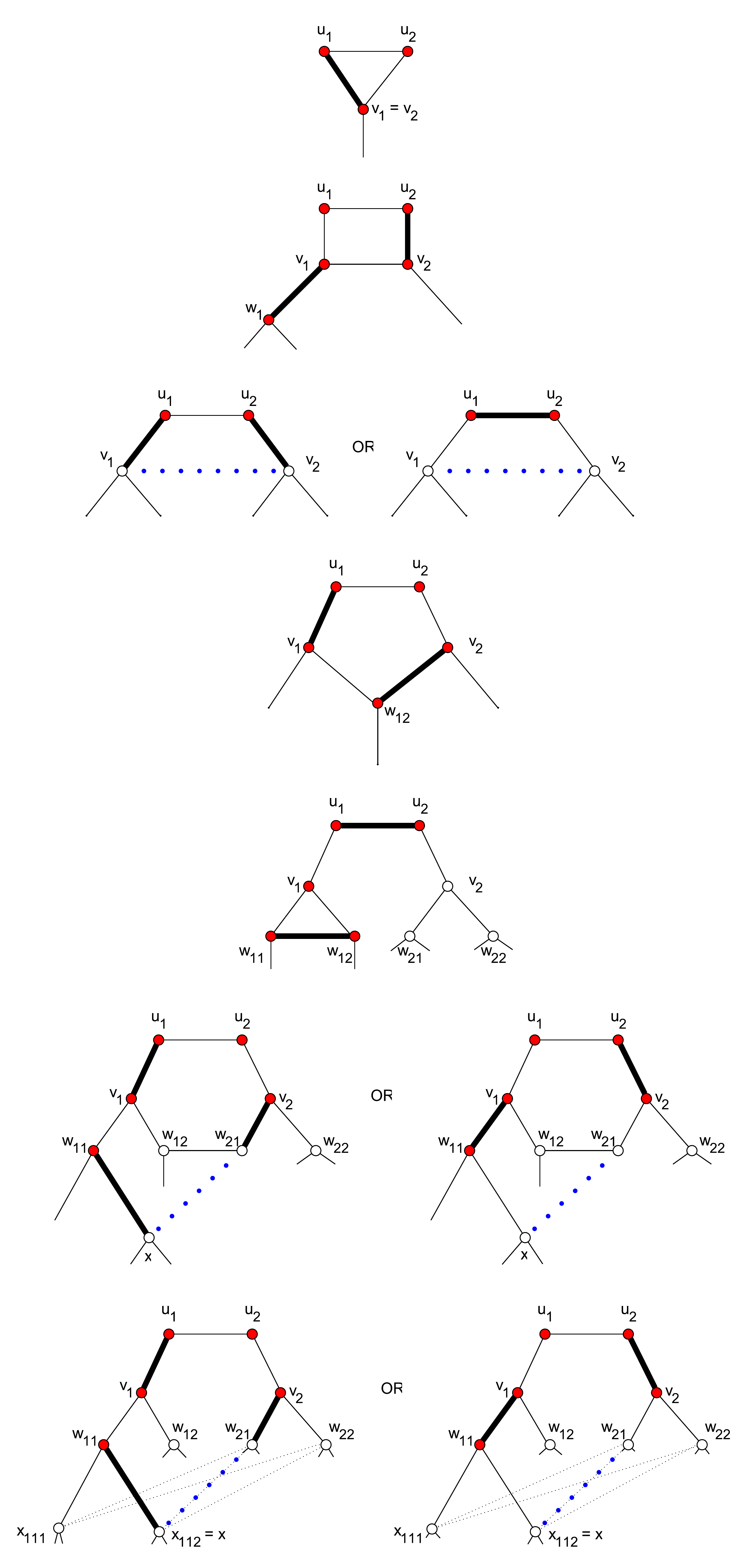}
\caption{The case analysis in Lemma~\ref{lem:noadjacentdegreetwos}, from top to bottom. In each case, to obtain the reduced graph $G'$, the red vertices are removed and the blue dotted line is added. The fat edges extend the maximal matching from $G'$ to $G$. The small-dotted lines in the final case represent $E^*$, the vertex pairs of which at least one is not an edge.}
\label{fig:Lemma11}
\end{figure}

\begin{lemma}\label{lem:noadjacentdegreetwos}
There are no two adjacent degree-$2$ vertices in $G$.
\end{lemma}
\begin{proof}
See Figure~\ref{fig:Lemma11} for a depicttion of the case analysis for this lemma. Suppose for a contradiction that there are two adjacent degree-$2$ vertices $u_1$ and $u_2$. Let $v_1$ be the other neighbour of $u_1$ and let $v_2$ be the other neighbour of $u_2$. First, if $v_1=v_2$ then consider $G'=G-\left\{u_1,u_2, v_1\right\}$. Then a maximal matching $M'$ of $G'$ can be extended to a maximal matching $M' + u_1v_1$ of $G$, so $\delta_{\gamma}\leq 1$. Therefore
$\delta_{\lb} \geq \frac{4 \delta_n - \delta_m}{6} \geq \frac{4 \cdot 3 - 4}{6} > 1 \geq \delta_{\gamma}.$

So we may assume that $v_1\neq v_2$. Next, suppose that $v_1v_2 \in E(G)$. Since the graph induced by $\left\{u_1,u_2,v_1,v_2 \right\}$ is not a counterexample, without loss of generality $v_1$ has a neighbour $w_1$ distinct from $u_1,v_2$. Define $G'=G-\left\{u_1,u_2, v_1, v_2, w_1\right\}$. Then a maximal matching $M'$ of $G'$ can be extended to a maximal matching $M'+u_2v_2+v_1w_1$ of $G$, so $\delta_{\gamma} \leq 2$. So

$$\delta_{\lb} \geq \frac{4 \delta_n - \delta_m}{6} \geq \frac{4 \cdot 5 - 8}{6} =2 \geq \delta_{\gamma}.$$

Therefore $v_1v_2 \notin E(G)$. Next, consider the graph $G'= G- \left\{  u_1,u_2 \right\} + v_1v_2$. (Equivalently: $G'$ is obtained from $G$ by contracting the edges $u_1v_1$ and $u_2v_2$.) 
Let $M'$ be a maximal matching of $G'$. If $v_1v_2 \in M'$, then $M'-v_1v_2+u_1v_1+u_2v_2$ is a maximal matching of $G$. On the other hand, if $v_1v_2 \notin M'$, then $M'+u_1u_2$ is a maximal matching of $G$. In both cases the new matching of $G$ has size $|M'|+1$, so we have $\delta_{\gamma} \leq 1$. Furthermore, $\delta_n=2$ and $\delta_m=2$. 
If $G'$ is not cubic, then $\delta_I = 0$ so it follows that 
$$\delta_{\lb} \geq \frac{4 \cdot \delta_n - \delta_m}{6} = \frac{4 \cdot 2 - 2}{6} = 1 \geq \delta_{\gamma},$$
a contradiction. Thus, $G'$ must be cubic. Therefore from now on we may assume that all vertices of $G$ other than $u$ and $v$ have degree three. In particular, $v_1$ has neighbours $w_{11}, w_{12}$ and $v_2$ has neighbours $w_{21}, w_{22}$ such that $\left\{w_{11},w_{12},w_{21}, w_{22}\right\}$ is disjoint from $\left\{u_1,u_2,v_1,v_2  \right\}$.

Next, suppose that $v_1$ and $v_2$ have a common neighbour, without loss of generality it is $w_{12}$. Then a maximal matching $M'$ of $G':=G-\left\{u_1,u_2,v_1,v_2,w_{12} \right\}$ can be extended to the maximal matching $M'+u_1v_1 + w_{12}v_2$ of $G$, which implies that $\delta_{\gamma} \leq 2$. Thus again $\delta_{\lb} \geq \frac{4 \cdot \delta_n - \delta_m}{6} \geq \frac{4 \cdot 5 - 8}{6} =2 \geq \delta_{\gamma}.$ We conclude that $w_{11},w_{12},w_{21}, w_{22}$ must be pairwise distinct. Our next goal is to show that these four vertices form an independent set.

First, assume that $w_{11}w_{12}$ or $w_{21}w_{22}$ is an edge, without loss of generality the former. Then a maximal matching $M'$ of $G'=G-\left\{u_1,u_2,v_1,w_{11},w_{12}\right\}$ yields a maximal matching $M'+u_1u_2+w_{11}w_{12}$ of $G$, so that $\delta_{\gamma}\leq 2$. Thus $\delta_{\lb} \geq \frac{4\delta_n-\delta_m}{6}\geq \frac{4\cdot 5-8}{6}=2 \geq \delta_{\gamma}$; contradiction.

Second, assume that there is an edge joining $\left\{w_{11},w_{12}\right\}$ and $\left\{w_{21},w_{22}\right\}$; without loss of generality $w_{12}w_{21}$ is such an edge.
Observe that then at least one neighbour $x$ of $w_{11}$ is not equal to $v_1$ or $w_{21}$, nor adjacent to $w_{21}$ (because otherwise $w_{21}$ would have more than three distinct neighbours). We define the reduced graph $G'=G-\left\{u_1,u_2,v_1,v_2,w_{11}\right\} + xw_{21}$. Let $M'$ be a maximal matching of $G'$. If $xw_{21} \in M'$, then $M'-x w_{21} + xw_{11} + w_{21}v_2 + u_1v_1$ is a maximal matching of $G$, and otherwise $M'+v_1w_{11}+u_2v_2$ is a maximal matching of $G$. In both cases the new matching contains two more edges than $M'$, so $\delta_{\gamma}\leq 2$. Furthermore, $\delta_{n}=5$ and $\delta_{m}\geq 9-1=8$ (here we use that $xw_{21}$ is an edge in $G'$ but not in $G$). Crucially, $G'$ does not contain any cubic component, so that $\delta_{I}\geq 0$. Indeed, the only component of $G'$ that could a priori be cubic is the component $C$ containing the added edge $xw_{21}$. But due to the edge $w_{12}w_{21}$ (which we assumed to exist), we know that $C$ also contains the vertex $w_{12}$, which has degree less than $3$ in $G'$.
We obtain $\delta_{\mu} \geq \frac{4 \delta_n-\delta_m + 2\delta_{I}}{6} \geq 2 \geq \delta_{\gamma}$; contradiction. This concludes the proof that $\left\{w_{11},w_{12},w_{21},w_{22}\right\}$ is independent.

In summary, we have deduced so far that $\left\{u_1,u_2,v_1,v_2,w_{11},w_{12},w_{21},w_{22}  \right\}$ induces a tree in $G$, and all vertices of $G$ other than $u_1$ and $u_2$ have degree three.
%For all $i,j \in \left\{1,2\right\}$, let $x_{ij1}$ and $x_{ij2}$ denote the two neighbours of $w_{ij}$ that are distinct from $v_i$. 
%One should keep in mind that $x_{111},x_{112},x_{121},\ldots, x_{222}$ are not necessarily pairwise distinct.
Let $x_{111}$ and $x_{112}$ denote the two neighbours of $w_{11}$ that are distinct from $v_1$.

Now let us define the vertex set $V_0:=\left\{u_1,u_2,v_1,v_2,w_{11}\right\}$ and the set $E^*:=\left\{x_{111}w_{21},x_{112}w_{21},x_{111}w_{22},x_{112}w_{22}\right\}$ which consists of pairs of neighbours of $V_0$. If each pair in $E^*$ is present as an edge in $G$, then it follows that $v_1w_{12}$ is a bridge, contradicting Lemma~\ref{lem:nobridge}. Thus we can apply Lemma~\ref{lem:technicallemma} to conclude that there exists $e\in E^*$ such that $e\notin E(G)$ and the reduced graph $G'=G-V_0+e$ has no cubic component. This means that $\delta_I=0$ with respect to $G'$. Furthermore, by exactly the same analysis as before (by symmetry of $G$, we may assume that $e=xw_{21}$ for some neighbour $x$ of $w_{11}$), we again obtain $\delta_{\gamma}\geq 2, \delta_n=5$ and $\delta_m\geq 9-1=8$. Therefore $\delta_{\mu} \geq \frac{4 \delta_n-\delta_m}{6} \geq 2 \geq \delta_{\gamma}$; contradiction.

\end{proof}

\begin{lemma}\label{lem:nocubicneighbours}
Every degree-$2$ vertex of $G$ has at most one degree-$3$ neighbour.
\end{lemma}
\begin{proof}
We suppose for a contradiction that a degree-$2$ vertex $u$ of $G$ has two degree-$3$ neighbours $v_1$ and $v_2$. 

If $v_1v_2\in E(G)$, then every maximal matching $M'$ of $G'=G-\left\{u,v_1,v_2\right\}$ can be extended to a maximal matching $M'+v_1v_2$ of $G$, so $\delta_{\gamma}\leq 1$. So $\delta_{\lb} \geq \frac{4 \delta_n-\delta_m}{6} \geq \frac{4 \cdot 3-5}{6}> 1 \geq \delta_{\gamma}$; contradiction.
Therefore $v_1v_2\notin E(G)$. So $v_1$ has neighbours $w_{11}, w_{12}$ and $v_2$ has neighbours $w_{21}, w_{22}$ such that $\left\{w_{11},w_{12},w_{21}, w_{22}\right\}$ is disjoint from $\left\{u,v_1,v_2 \right\}$.

\textbf{Case $1$: $v_1$ and $v_2$ have a common neighbour that is distinct from $u$.}\\
 Without loss of generality this common neighbour is $w_{12}=w_{21}$. Consider the graph $G'=G-\left\{u,v_1,v_2,w_{11},w_{12}\right\}$ with a maximal matching $M'$. Then $M'+v_1w_{11}+v_2w_{12}$ is a maximal matching of $G$, so $\delta_{\gamma}\geq 2$. Moreover, $\delta_n=5$ and $\delta_{m}\leq 9$. If $\delta_m\leq 8$, then it follows that $\delta_{\lb}\geq 2 \geq \delta_{\gamma}$. Thus $\delta_m$ must equal $9$, which implies that $w_{11},w_{12},w_{22}$ are distinct vertices, that $w_{11}$ and $w_{12}$ have degree three and that $w_{11}w_{12} \notin E(G)$. By symmetry, $w_{22}$ must have degree three as well, and $w_{22}w_{12}\notin E(G)$. Thus, $w_{11},w_{12},w_{22}$ are distinct degree-$3$ vertices that either form an independent set or induce the graph with the edge $w_{11}w_{22}$.
 
 In particular, $w_{12}$ has a neighbour $x_{12}$ which is distinct from $u,v_1,v_2,w_{11},w_{22}$.
For $i \in \left\{1,2\right\}$, let $x_{ii1}$ and $x_{ii2}$ be the two neighbours of $w_{ii}$ that are distinct from $v_i$. 

First assume that $x_{12}$ is a common neighbour of $w_{11}$ and $w_{22}$. If additionally $w_{11}w_{22} \in E(G)$, then the structure of $G$ is fully determined; it is the graph induced by the seven vertices $u,v_1,v_2,w_{11},w_{12},w_{22},x_{12}$, which clearly is not a counterexample. Thus $w_{11}w_{22}\notin E(G)$ and so $w_{11}$ has a neighbour $x_{111}$ that is distinct from $v_1, x_{12}$ and $w_{22}$. In that case, a maximal matching $M'$ of $G'=G-\left\{u,v_1,v_2,w_{11},w_{12},w_{22},x_{12},x_{111}\right\}$ can be extended to the maximal matching $M'+w_{11}x_{111}+v_1w_{12}+v_2w_{22}$ of $G$, so that $\delta_{\gamma} \leq 3$, $\delta_n=8$ and $\delta_m\leq 13$. Hence $\delta_{\lb} \geq \frac{4 \cdot 8 - 13}{6} > 3 \geq \delta_{\gamma}$; contradiction. We have thus derived that $x_{12}$ is not a neighbour of both $w_{11}$ and $w_{22}$.

%\todo[inline]{Here we can probably use the technical lemma as well. Oh, actually not directly.}
By symmetry, we may assume that $x_{12}$ is not a neighbour of $w_{22}$. We then consider the graph $G'=G-\left\{ u,v_1,v_2,w_{11},w_{12}\right\} + x_{12}w_{22}$.
 Then $\delta_n=5$. Note that to obtain $G'$ from $G$, we not only delete vertices, but we also add an edge, so $\delta_m \leq 9-1=8$. Consider a maximal matching $M'$ of $G'$. If $x_{12}w_{22} \in M'$ then  $M'-x_{12}w_{22}+x_{12}w_{12}+v_{2}w_{22} + v_1w_{11}$ is a maximal matching of $G$, and otherwise $M'+v_2w_{12}+v_1w_{11}$ is a maximal matching of $G$. In both cases $\delta_{\gamma}\leq 2$. To arrive at a contradiction, we still need to check that $G'$ has no cubic component (so that $\delta_I=0$). For this, it suffices to demonstrate that the component $C$ of $G'$ containing the added edge $x_{12}w_{22}$ is not cubic. 
 If $w_{11}w_{22}\in E(G)$, then it is immediate that $w_{22}$ is in $V(C)$ and has degree less than three in $G'$. Thus we may assume that $w_{11}w_{22}\notin E(G)$.
Because $G$ has no bridge (in particular $w_{11}v_1$ is not a bridge), there exists a path $P$ in $G$ that joins $\left\{x_{111},x_{112}\right\}$ with $\left\{x_{12},w_{22}\right\}$ while avoiding $\left\{w_{11},u,v_1,v_2,w_{12}\right\}$. (Here we allow $P$ to be a single vertex.) Note that $P$ is also a path in $G'$ and hence $x_{111} \in V(C)$ or $x_{112} \in V(C)$. Since $x_{111}$ and $x_{112}$ have degree less than three in $G'$, it follows that $C$ is not cubic.

\textbf{Case $2$: $u$ is the only common neighbour of $v_1$ and $v_2$.}

From now on, we know that $w_{11},w_{12},w_{21},w_{22}$ are pairwise distinct. Our next task is to show that they form an independent set. 

\textbf{Case $2.1$: $w_{11}w_{12} \in E(G)$ or $w_{21}w_{22} \in E(G)$.}

By symmetry, it suffices to consider the case that $w_{11}w_{12} \in E(G)$. Suppose for a contradiction that $w_{11}w_{12} \in E(G)$. 

If additionally there is an edge between $\left\{w_{11},w_{12}\right\}$ and $\left\{w_{21},w_{22}\right\}$  (say $w_{12}w_{21}$ is an edge), then a maximal matching $M'$ of $G'=G-\left\{u,v_1,v_2,w_{11},w_{12},w_{21}  \right\}$ can be extended to the maximal matching $M'+v_1w_{11}+v_2w_{21}$ of $G$. Then $\delta_{\lb}\geq \frac{4 \delta_n-\delta_m}{6} \geq \frac{4\cdot 6-9}{6} > 2 \geq \delta_{\gamma}$; contradiction. Thus there is no edge between $\left\{w_{11},w_{12}\right\}$ and $\left\{w_{21},w_{22}\right\}$.

By considering the graph $G'=G-\left\{u,v_1,v_2,w_{11},w_{12}  \right\}$, it is easily seen that each of $w_{11}, w_{12}$ needs to have degree three; otherwise $\delta_m\leq 8$, leading to the contradiction $\delta_{\lb} \geq \frac{4 \delta_n-\delta_m}{6} \geq \frac{4\cdot 5-8}{6}= 2 \geq \delta_{\gamma}$, where the last inequality holds because any maximal matching $M'$ of $G'$ can be extended to the maximal matching $M'+w_{11}w_{12}+uv_2$ of $G$. 

For $i\in \left\{1,2 \right\}$, let $x_{1i}$ denote the neighbour of $w_{1i}$ that is distinct from $v_1,w_{11},w_{12}$.

Let $V_0:=\left\{u,v_1,v_2,w_{11},w_{12}\right\}$ and note that $N:= \left\{ x_{11},x_{12},w_{21},w_{22}\right\}$ is the set of neighbours of $V_0$. (Possibly $N$ has less than four distinct elements since it could be that $x_{11}=x_{12}$.) 
If $\left\{x_{11},x_{12}\right\}$ is completely connected to $\left\{w_{21},w_{22}\right\}$, then $G$ must be the graph induced by $u,v_1,v_2,w_{11},w_{12},w_{21},w_{22},x_{11},x_{12}$, which is a graph on nine vertices; contradiction. Thus there must exist $e\in E^*:=\left\{x_{11}w_{21},x_{11}w_{22},x_{12}w_{21},x_{12}w_{22} \right\}$ such that $e \notin E(G)$. 
Then by Lemma~\ref{lem:technicallemma} applied to $V_0$ and $E^*$, there exists an edge $e\in E^*$ such that the graph $G'=G-V_0+e$ has no cubic component. By symmetry, we may assume that $e=x_{11}w_{22}$. 
By the choice of $e$ we have $\delta_I=0$. Moreover $\delta_n=5$ and $\delta_m\leq 8$. Let $M'$ be a maximal matching of $G'$. If $x_{11}w_{22}\in M'$, then $M'-x_{11}w_{22}+x_{11}w_{11}+v_2w_{22}+v_1w_{12}$ is a maximal matching of $G$, and otherwise $M'+w_{11}w_{12}+uv_2$ is a maximal matching of $G$. In both cases, $\delta_{\gamma}\leq 2$. Thus we arrrive at the contradiction $\delta_{\lb}\geq \frac{4\delta_n-\delta_m+2\delta_I}{6}=\frac{4\cdot 5-8}{6}=2 \geq \delta_{\gamma}$.

\textbf{Case $2.2$: $w_{11}w_{12}, w_{21}w_{22}\notin E(G)$, but $w_{11}, w_{12}, w_{21}, w_{22}$ do not form an independent set.}
%, and at least one edge joins $\left\{w_{11},w_{12}\right\}$ and $\left\{w_{21},w_{22}\right\}$.}
%at least one of $w_{12}w_{21}, w_{11}w_{22}$ is in $E(G)$.}

Without loss of generality, we may assume that $w_{12}w_{21}\in E(G)$. By considering the graph $G'=G-\left\{u,v_1,v_2,w_{12},w_{21}\right\}$, it quickly follows that $w_{12}$ and $w_{21}$ must have degree three.
So let $x_{12}$ (respectively $x_{21}$) denote the neighbour of $w_{12}$ (respectively $w_{21}$) that is distinct from $v_1,v_2,w_{12},w_{21}$.

First suppose that additionally $w_{11}w_{22}, w_{11}x_{12}, w_{22}x_{21} \in E(G)$. If $x_{12}$ and $x_{21}$ are adjacent or both have degree two, then $G$ is a graph on nine vertices and it is easy to see that it is not a counterexample. So we may assume that $x_{21}$ has a third neighbour $y$ distinct from $w_{21}$ and $w_{22}$. In that case a maximal matching $M'$ of $G-\left\{u,v_1,v_2,w_{11},w_{12},w_{21},w_{22},x_{12},x_{21},y \right\}$ can be extended to the maximal matching $M'+w_{11}w_{22}+uv_2+w_{12}x_{12}+x_{21}y$ of $G$. Since $x_{12} \neq x_{21}$ (otherwise this vertex would have degree four), it follows that $\delta_n=10$. Furthermore $\delta_m \leq 16$, so $\delta_\mu \geq \frac{4 \cdot 10 -16}{6}=4 \geq \delta_{\gamma}$; contradiction. We conclude that at least one of $w_{11}w_{22},w_{11}x_{12},w_{22}x_{21}$ is not an edge of $G$. 

From Lemma~\ref{lem:technicallemma} applied with $V_0=\left\{u,v_1,v_2,w_{12},w_{21}\right\}$ and $E^*=\left\{ w_{11}w_{22},w_{11}x_{12},w_{22}x_{21}\right\}$, we obtain an $e\in E^*$ such that $G'=G-V_0+e$ has no cubic component.
Let $M'$ be a maximal matching of $G'$.
First, suppose that $e=w_{11}w_{22}$. If $w_{11}w_{22} \in M'$, then $M'-w_{11}w_{22}+v_1w_{11}+v_{2}w_{22}+w_{12}w_{21}$ is a maximal matching of $G$, and otherwise $M'+v_1w_{12}+v_2w_{21}$ is a maximal matching of $G'$. So in both cases $\delta_{\gamma}\leq 2$.
Second, suppose that that $e=w_{11}x_{12}$ (the case that $e=w_{22}x_{21}$ is symmetric). If $w_{11}x_{12} \in M'$, then $M'-w_{11}x_{12}+v_1w_{11}+x_{12}w_{12}+v_{2}w_{21}$ is a maximal matching of $G$, and otherwise $M'+v_1w_{12}+v_2w_{21}$ is a maximal matching of $G$. So in all cases we have $\delta_{\gamma}\leq 2$.
For all possible choices of $e$, we have $\delta_n=5$ and $\delta_m\leq 9-1=8$. Since no component of $G'$ is cubic, we conclude that $\delta_{\lb}\geq \frac{4\delta_n-\delta_m+2\delta_I}{6}\geq \frac{4\cdot 5-8+0}{6}=2 \geq \delta_{\gamma}$; contradiction.

\begin{figure}[h]
\centering
\includegraphics[width=0.66\textwidth]{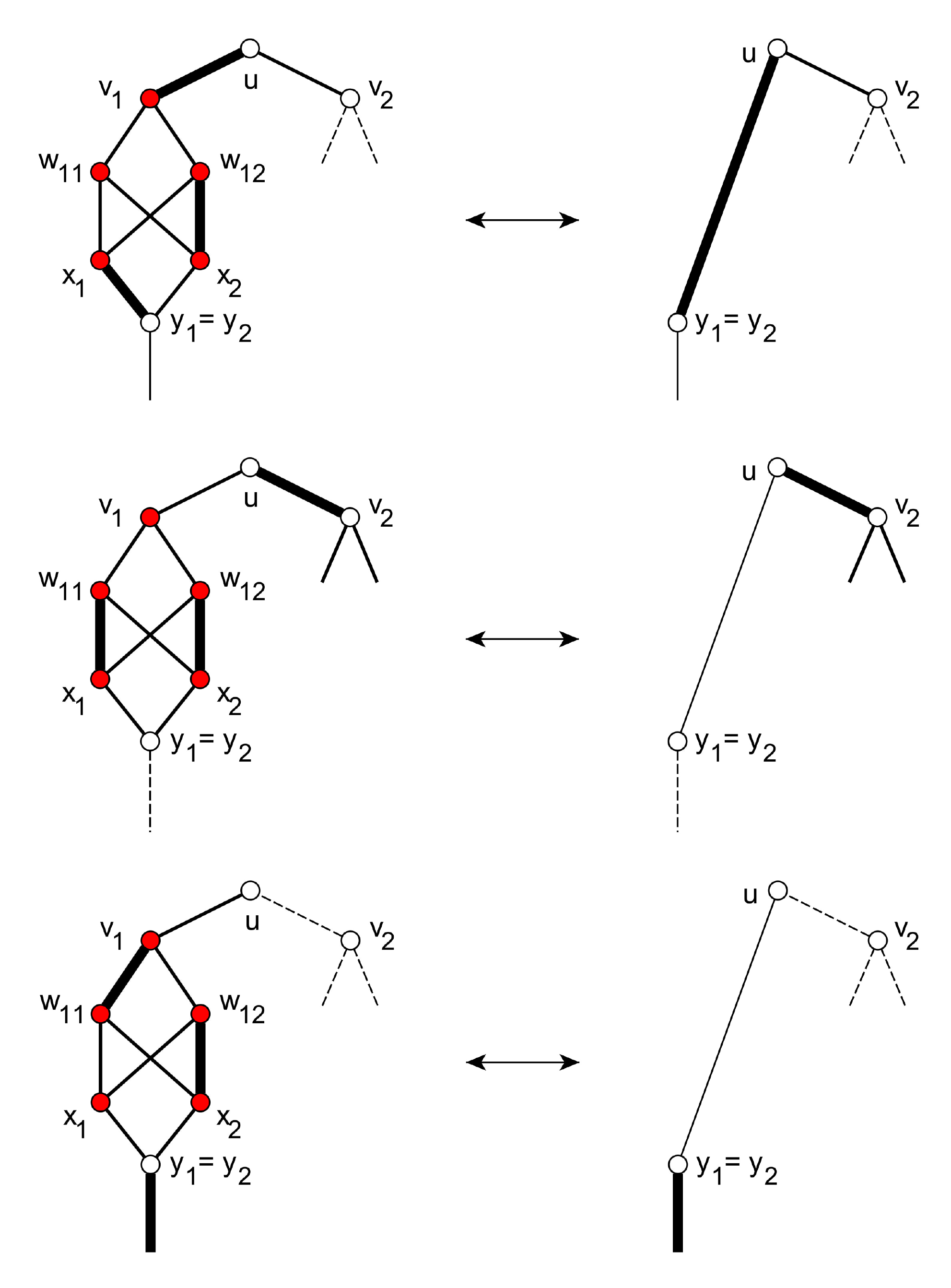}
\caption{The end of Case $2.3.1$. The figures on the left represent $G$, the pictures on the right $G'$. From top to bottom, these are three possible scenarios in which a maximal matching $M'$ of $G'$ can be extended to a maximal matching $M$ of $G$ with $|M|=|M'|+2$. The fat lines represent edges that are in the maximal matching, while the dashed lines are edges that may or may not be in the maximal matching.}
\label{fig:case2_3_1}
\end{figure}

\textbf{Case 2.3: $w_{11},w_{12}, w_{21},w_{22}$ form an independent set.}

So far, we have derived that $u,v_1,v_2,w_{11},w_{12},w_{21},w_{22}$ are distinct vertices and induce a tree.

\textbf{Case 2.3.1: For some $i \in \left\{ 1,2\right\}$,  $w_{i1}$ and $w_{i2}$ have three common neighbours.}

By symmetry, it suffices to consider the case $i=1$. We already know that $v_1$ is a common neighbour of $w_{11}$ and $w_{12}$. Suppose they have two other common neighbours, say $x_{1}$ and $x_{2}$. By the structure derived so far, $x_1$ and $x_2$ are distinct from $u,v_1,v_2,w_{11},w_{12},w_{21},w_{22}$.

If $x_1$ and $x_2$ are either adjacent or both have degree two, then $uv_1$ is a bridge, contradicting Lemma~\ref{lem:nobridge}. 
If one of $x_1,x_2$ (say $x_2$) has degree two then a   maximal matching $M'$ of $G'=G-\left\{ u,v_1,v_2,w_{11},w_{12},w_{22},x_1,x_2 \right\}$ yields the maximal matching $M'+w_{11}x_{1}+v_1w_{12}+v_2w_{22}$ of $G$, so $\delta_{\lb} \geq \frac{4 \delta_n-\delta_m}{6}\geq \frac{4 \cdot 8 -13}{6} > 3 \geq \delta_{\gamma}$, contradiction. Thus both $x_1$ and $x_2$ have degree three and are nonadjacent. For $i\in \left\{1,2\right\}$, let $y_i$ denote the neighbour of $x_i$ that is distinct from $w_{11}$ and $w_{12}$. 

If $y_1 \neq y_2$, then a  maximal matching $M'$ of $G'= G- \left\{ u,v_1,w_{11},w_{12},x_1,x_2,y_1,y_2\right\}$ yields the maximal matching $M'+uv_1+x_1y_1+x_2y_2$ of $G$, so $\delta_{\lb} \geq \frac{4 \delta_n-\delta_m}{6}\geq \frac{4 \cdot 8 -14}{6} \geq 3 \geq \delta_{\gamma}$, contradiction. 

Thus $y_1=y_2$. See Figure~\ref{fig:case2_3_1}. Consider the graph $G'=G-\left\{ v_1,w_{11},w_{12},x_1,x_2 \right\} + uy_1$. Note that $G'$ has no cubic component (indeed, the vertex $u$ has degree two in $G'$), so $\delta_I=0$. Furthermore, $\delta_n=5$ and $\delta_m \leq 9-1=8$.  Let $M'$ be a maximal matching of $G'$. If $uy_1\in M'$ then $M'-uy_1+uv_1+x_1y_1+w_{12}x_2$ is a maximal matching of $G$. Otherwise, if $uy_1 \notin M'$, we need to do something slightly more involved than before: since $M'$ is maximal and $uy_1 \notin M'$, we must either have $uv_2 \in M'$ (in which case $M'+w_{11}x_1+w_{12}x_2$ is a maximal matching of $G$) or an edge incident to $y_1$ is in $M'$ (in which case $M'+v_1w_{11}+w_{12}x_2$ is a maximal matching of $G$). Thus in all cases $\delta_{\gamma} \leq 2$, so $\delta_{\lb} \geq \frac{4\cdot 5-8}{6} \geq 2 \geq \delta_{\gamma}$; contradiction.

\textbf{Case 2.3.2 For $i \in \left\{1,2\right\}$, $w_{i1}$ and $w_{i2}$ do not have three common neighbours.}

%First, suppose that both $\left\{w_{11},w_{12}\right\}$ and $\left\{w_{21},w_{22}\right\}$ contain a degree-$2$ vertex, say $w_{12}$ and $w_{21}$ have degree two. Then for $G'=G-\left\{u,v_1,v_2,w_{12},w_{21}\right\}$ we obtain that $\delta_m$ is at most $8$, so $\delta_{\lb} \geq \frac{4\delta_n-\delta_m}{6} \geq \frac{4\cdot 5-8}{6}=2 \geq \delta_{\gamma}$, contradiction.

First, suppose that at least one of $\left\{w_{11},w_{12}\right\}$ and $\left\{w_{21},w_{22}\right\}$ contains no degree-$3$ vertex; say $w_{11}$ and $w_{12}$ both have degree two. Then for $i\in \left\{1,2\right\}$, let $x_{1i}$ be the unique neighbour of $w_{1i}$ that is distinct from $v_1$.
 If $x_{11}=x_{12}$, then for $G'= G-\left\{u,v_1,w_{11},w_{12},x_{11} \right\}$ we have that $M'+uv_1+w_{12}x_{12}$ is a maximal matching of $G$ for any maximal matching $M'$ of $G'$, so $\delta_{\lb} \geq \frac{4\delta_n-\delta_m}{6} \geq \frac{4\cdot 5-7}{6}>2 \geq \delta_{\gamma}$; contradiction.
  On the other hand, if $x_{11}\neq x_{12}$, then $w_{11}$ and $x_{12}$ are not adjacent in $G$, so we can consider $G'=G-\left\{u,v_1,v_2,w_{12},w_{21}\right\}+w_{11}x_{12}$. Let $M'$ be a maximal matching of $G$. If $w_{11}x_{12}$ is in $M'$ then $M'-w_{11}x_{12}+v_1w_{11}+w_{12}x_{12}+v_2w_{21}$ is a maximal matching of $M$, and otherwise $M'+v_1w_{12}+v_2w_{21}$ is. Since $w_{11}$ has degree two in $G$, it has degree two in $G'$ as well, and hence $G'$ has no cubic component, so $\delta_I=0$. Thus $\delta_{\lb}  \geq \frac{4\delta_n-\delta_m}{6} \geq \frac{4\cdot 5-8}{6} =2 \geq \delta_{\gamma}$; contradiction.

Therefore we may henceforward assume, without loss of generality, that $w_{12}$ and $w_{21}$ have degree three. 
  Let $x_{121}$ and $x_{122}$ denote the two neighbours of $w_{12}$ that are distinct from $v_1$. Similarly, let $x_{211}$ and $x_{212}$ denote the two neighbours of $w_{21}$ that are distinct from $v_2$.

Now the real fun begins. Let $V_0:=\left\{ u,v_1,v_2,w_{12},w_{21} \right\}$ and consider its neighbours $N=\left\{w_{11},x_{121},x_{122},x_{211},x_{212},w_{22} \right\}$. Note that that these neighbours are not necessarily all distinct, so in what follows, we treat $N$ as a multi-set. \footnote{We could also avoid treating $N$ as a multi-set by instead doing a case analysis on $N$ containing $4,5$ or $6$ distinct vertices of $G$. This would however essentially necessitate repeating the same long argument three times. When reading the current proof, the reader may find it convenient to have in mind the case that all elements of $N$ are distinct vertices of $G$.}

We first discuss two consequences from the fact that $G$ is bridgeless. For each $x\in N$, let $A(x)$ denote the component of $G-V_0$ containing $x$. Each $y_1\in N$ either has more than one neighbour in $V_0$ (in which case $y_1=y_2$ for some $y_2 \in N-\left\{y_1\right\}$) or has exactly one neighbour $v$ in $V_0$ (in which case the fact that $y_1v$ is not a bridge implies $A(y_1)=A(y_2)$ for some $y_2 \in N-\left\{y_1\right\}$). Thus: 
\begin{equation}\label{prop:single}
\text{ For every } y_1 \in N \text{ there exists } y_2 \in N-\left\{y_1\right\} \text{ such that } A(y_1)=A(y_2).
\end{equation}

Furthermore, in particular using that $uv_1$ is not a bridge, we obtain 
\begin{equation}\label{prop:cross}
A(y_1)=A(y_2) \text{ for some } y_1 \in \left\{w_{11},x_{121},x_{122}\right\} \text{ and some } y_2 \in \left\{w_{22},x_{211},x_{212}\right\}.
\end{equation}

Since we are working under the assumption of Case $2.3.2$, we know that there exist $q_1 \in \left\{x_{121},x_{122}\right\}$ and $q_2 \in \left\{ x_{211},x_{212} \right\}$ such that
\begin{equation}\label{eq:nonadjacency}
w_{11}q_1 \notin E(G) \text{ and } w_{22}q_2 \notin E(G).
\end{equation}

Given a choice of such $q_1,q_2$, we consider the graph $G'=G-V_0+w_{11}q_1+w_{22}q_2$. 
Our main task is to show that $q_1$ and $q_2$ can be chosen such that additionally $G'$ has no cubic component (see Figure~\ref{fig:case2_3_2}). Once we have established that, we can use that $G'$ satisfies Theorem~\ref{thm:maincubic}, as we will do at the very end.

 Given a choice of  $q_1 \in \left\{x_{121},x_{122}\right\}, q_2 \in \left\{ x_{211},x_{212} \right\}$ satisfying (\ref{eq:nonadjacency}) and a vertex $x\in N$, we let $C(x)$ denote the component of $G'$ that contains $x$. If $x$ is distinct from $w_{11},q_1,w_{22},q_2$ then $x$ has degree less than three in $G'$, so it is immediate that $C(x)$ is not cubic. Therefore it suffices to show that $C(w_{11})$ and $C(w_{22})$ are not cubic; in particular, it suffices to show that 
\begin{equation}\label{prop:atleast}
\text{both } C(w_{11}) \text{ and } C(w_{22}) \text{ contain at least one vertex of } N - \left\{ w_{11},w_{22},q_1,q_2 \right\}.
\end{equation}
 To demonstrate that there is indeed a choice of $q_1,q_2$ satisfying (\ref{eq:nonadjacency}) and (\ref{prop:atleast}), let us introduce the auxiliary graph $H$ on the vertex set $N$ (with six elements) in which two vertices are adjacent if and only if they belong to the same component of $G-V_0$. (In particular $a,b \in N$ are adjacent in $H$ if $a$ and $b$ are the same vertex in $G$.) Furthermore, let $H^{+}$ be the graph obtained from $H$ by adding the two edges $w_{11}q_1$ and $w_{22}q_2$, if they were not already present in $H$. To satisfy requirement~(\ref{prop:atleast}), it suffices to show that $H^{+}$ contains a component with at least five vertices, for some choice of $q_1$ and $q_2$.

By property~(\ref{prop:single}), $H$ has minimum degree at least one. By property~(\ref{prop:cross}), $H$ contains an edge $y_1y_2$ for some $y_1 \in N_1:=\left\{w_{11},x_{121},x_{122}\right\}$ and $y_2 \in N_2:= \left\{ w_{22},x_{211},x_{212}\right\}$. 

Suppose first that $N_1$ and $N_2$ are independent sets in $H$. Then $H$ must contain three edges that form a matching $M_H$ between $N_1$ and $N_2$. If $w_{11}$ and $w_{22}$ are adjacent in $M_H$, then choose $q_1$ and $q_2$ to be nonadjacent (with respect to $M_H$), and otherwise choose $q_1$ and $q_2$ to be adjacent (in $M_H$). In both cases, all six vertices of $H^+$ belong to the same component, as desired.

Thus we may assume that at least one of $H[N_1]$ and $H[N_2]$ has a connected component $C$ on at least two vertices, say $H[N_1]$ has. Then there is a choice of $q_1$ such that all three vertices of $N_1$ belong to the same component of $H+w_{11}q_1$. Furthermore, if $y_2\neq w_{22}$ and $y_2w_{22} \notin E(H)$ then choose $q_2=y_2$, and otherwise choose $q_2$ to be an arbitrary element of $\left\{x_{211}, x_{212} \right\}$ for which $w_{22}q_2\notin E(G)$. This ensures that in $H^+$, all five vertices of $N_1 \cup  \left\{q_2\right\} \cup \left\{w_{22}\right\}$ belong to the same component, as desired.

\begin{figure}[h]
\centering
\includegraphics[width=0.95\textwidth]{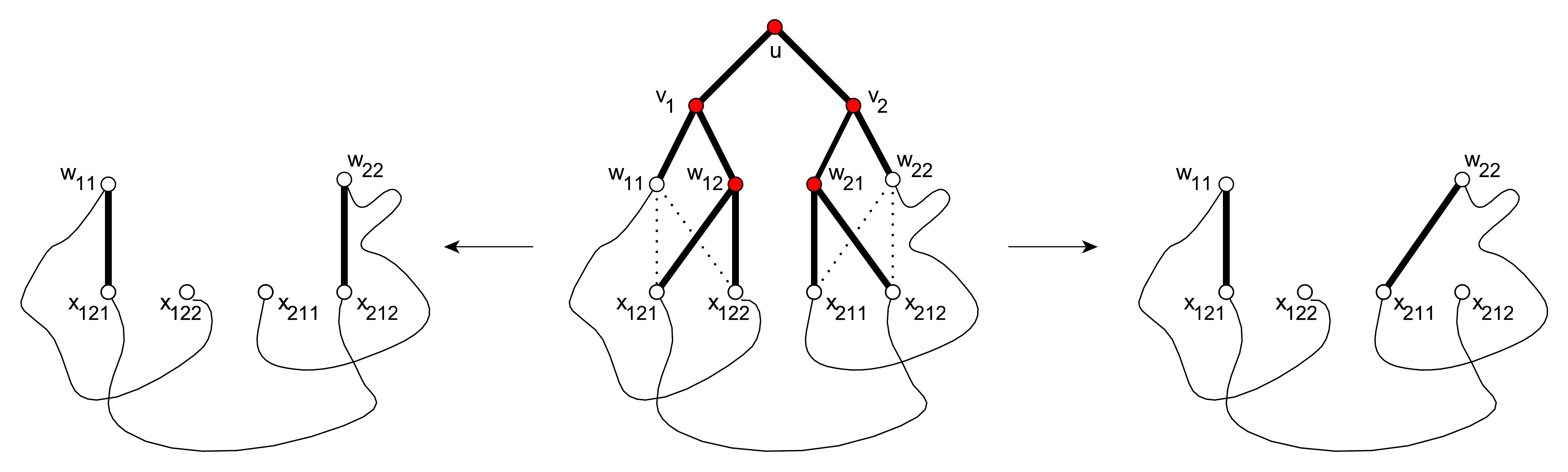}
\caption{The structure in Case $2.3.2$. In the middle is an example of the graph $G$, depicting one possible way in which the neighbours of $V_0=\left\{u,v_1,v_2,w_{12},w_{21}\right\}$ can be connected in $G-V_0$. At least two of the dotted edges are not present in $G$. On the left and right are depicted two possible choices of added edges $w_{11}q_1, w_{22}q_2$ corresponding to two distinct possibilites for the reduced graph $G'=G-V_0+w_{11}q_1+w_{22}q_2$. The reduced graph on the right is not a good choice as there the component containing $w_{22}$ could be cubic. The reduced graph on the left is a good choice because there each component (in fact the only component in this case) contains $x_{122}$, which has degree less than three in $G'$.}
\label{fig:case2_3_2}
\end{figure}

This concludes the proof that there exists a choice of $q_1 \in \left\{x_{121},x_{122}\right\}$ and $q_2 \in \left\{x_{211},x_{212} \right\}$ such that $w_{11}q_1 \notin E(G)$ and $w_{22}q_2 \notin E(G)$ and the graph $G'=G-V_0+w_{11}q_1+w_{22}q_2$ has no cubic component. Furthermore, $G'$ also has maximum degree at most three (this follows from the fact that in $G$, all vertices of $N$ have at least one neighbour in $V_0$, while if $q_1=q_2$ then $q_1$  has at least two neighbours in $V_0$).

We are now finally ready to use that $G'$ satisfies Theorem~\ref{thm:maincubic}. By the analysis above we have $\delta_I=0$. Furthermore, $\delta_n=5$. In constructing $G'$ we have deleted ten edges and added back two edges, so $\delta_m= 10-2=8$. It remains to estimate $\delta_{\gamma}$. Let $M'$ be a maximal matching of $G'$.

Let $M$ be obtained from $M'$ by

\begin{itemize}
\item adding $v_1w_{12}$, if $w_{11}q_1 \notin M'$;
\item removing $w_{11}q_1$ and adding $v_1w_{11}$ and $w_{12}q_1$, otherwise.
\end{itemize}
and

\begin{itemize}
\item adding $v_2w_{21}$, if $w_{22}q_2 \notin M'$;
\item removing $w_{22}q_2$ and adding $v_2w_{22}$ and $w_{21}q_2$, otherwise.
\end{itemize}

Then $M$ is a maximal matching of $G$ of size $|M'|+2$, so $\delta_{\gamma} \leq 2$. In conclusion, $\delta_{\lb} \geq \frac{4\delta_n-\delta_m+2\delta_I}{6} = \frac{4 \cdot 5-8}{6}=2 \geq \delta_{\gamma}$, contradiction.
\end{proof}

\begin{comment}
\begin{figure}
\centering
\includegraphics[width=0.7\textwidth]{case2_3_2_bis.pdf}
\caption{Case $2.3.2$. On the left is $G$. The dashed lines may or may not be present in $G$. On the right is a possible realisation of the reduced graph $G'$. The fat lines represent a possible realisation of the maximal matchings.}
\label{fig:case2_3_2}
\end{figure}
\end{comment}

\textbf{The finishing blow}

%Lemmas~\ref{lem:nodegreeone},~\ref{lem:noadjacentdegreetwos} and~\ref{lem:nocubicneighbours} together imply that $G$ is a cubic graph.
%Choose two adjacent vertices $u$ and $v$ and consider the graph $G'= G- \left\{u, v\right\}$. Let $M'$ be a maximal matching of $G'$.
%Then $G + uv$ is a maximal matching of $G$, so $\delta_{\gamma} \leq 1$. Furthermore,  $\delta_n=2$ and $\delta_{m}=5$. Unlike $G$, our new graph $G'$ has no cubic component, so $\delta_I=1$. It follows that
%$$\delta_{\lb}= \frac{4 \cdot \delta_n - \delta_m + 3 \delta_I}{6} =  \frac{4 \cdot 2 - 5 + 3 \cdot 1}{6} = 1 \geq \delta_{\gamma},$$
%contradiction. This concludes the proof of Theorem~\ref{thm:maincubic}.

Lemmas~\ref{lem:nodegreeone},~\ref{lem:noadjacentdegreetwos} and~\ref{lem:nocubicneighbours} together imply that $G$ is a cubic graph. Choose two adjacent vertices $u_1$ and $u_2$. 
We first consider the case that $u_1$ and $u_2$ have a common neighbour $v$. In that case $v$ has degree $1$ in the reduced graph $G'= G- \left\{u_1, u_2\right\}$, so that $\delta_{n_1}\leq -1$. Unlike $G$, our new graph $G'$ has no cubic component, so $\delta_I=1$. Since $G$ is not isomorphic to $K_{4}$, no component of $G'$ is isomorphic to $K_2$, so $\delta_K=0$.  Any given maximal matching $M'$ of $G'$ can be extended to the maximal matching $M'+u_1u_2$ of $G$; therefore $\delta_{\gamma}\leq 1$. Moreover, $\delta_n=2$ and $\delta_m=5$. Everything together, we obtain that $\delta_{\lb} = \frac{4\delta_n-\delta_m-\delta_{n_1}+\delta_K+ 2\delta_I}{6} \geq \frac{4\cdot 2-5-(-1)+0+2\cdot 1}{6}= 1 \geq \delta_{\gamma}$; contradiction.

Therefore $u_1$ and $u_2$ have no common neighbour. Let $v_{11},v_{12}$ denote the two neighbours of $u_1$ that are distinct from $u_2$, and let $v_{21},v_{22}$ be the neighbours of $u_2$ that are distinct from $u_1$. Consider the sets $V_0=\left\{u_1,u_2\right\}$ and $E^*=\left\{ v_{11}v_{21},v_{11}v_{22},v_{12}v_{21},v_{12}v_{22}\right\}$. Since $G$ is not isomorphic to $K_{3,3}$ by assumption, it follows that at least one element of $E^*$ is not an edge in $G$. Hence by Lemma~\ref{lem:technicallemma} applied to $V_0$ and $E^*$, we obtain (without loss of generality) that $v_{11}v_{22}\notin E(G)$ and the reduced graph $G'=G-V_0 + v_{11}v_{22}$ has no cubic component. We therefore have $\delta_I=1$, and also $\delta_n=2$ and $\delta_m=5-1=4$. Moreover, $\delta_{\gamma}\leq 1$ because any given maximal matching $M'$ of $G'$ can be extended to the maximal matching $M' - v_{11} v_{22} + u_1 v_{11} + u_2 v_{22}$ of $G$, if $v_{11}v_{22}\in M'$, and to $M' + u_1 u_2$ otherwise. It follows that

$$\delta_{\lb}\geq\frac{4\delta_n-\delta_m+2\delta_I}{6}=\frac{4\cdot 2-4+2\cdot 1}{6} = 1 \geq \delta_{\gamma},$$

a contradiction. This concludes the proof of Theorem~\ref{thm:maincubic}.
%which constitutes our final contradiction.  

\qed

%\section{Conclusion}
%\section{A conjecture}

\textbf{Acknowledgements}

This research was supported by $\geq 43$ cups of coffee and an ARC grant from the Wallonia-Brussels Federation. The author thanks the two anonymous reviewers for their detailed comments and suggestions.

\bibliographystyle{abbrv}
\bibliography{bibliography}

\end{document}